\documentclass[pdftex]{amsart}

\usepackage{amsmath, amssymb, amsthm, upref}
\usepackage{graphicx}
\usepackage{subcaption}
\usepackage{listliketab}
\usepackage{longtable}
\usepackage{url}

\graphicspath{{Chara_Pic/}}

\theoremstyle{plain}
\newtheorem{thm}{Theorem}[section]

\newtheorem{cor}[thm]{Corollary}
\newtheorem{prop}[thm]{Proposition}
\newtheorem{rem}[thm]{Remark}

\theoremstyle{remark}

\DeclareMathOperator{\inter}{int}

\newcommand{\R}{\mathbb{R}}

\newcommand{\G}{\Gamma}
\newcommand{\D}{\mathcal{D}(\G)}
\newcommand{\A}{\mathcal{A}(\G)}

\newcommand{\bi}{\varphi}
\newcommand{\SSCS}{\mathcal{C}^{2}(S,\bi)}
\newcommand{\SSCC}{\mathcal{C}^{3}(S,\bi)}

\newcommand{\Bi}{\Phi}				
\newcommand{\Bib}{\Bi_{B}}			
\newcommand{\Biw}{\Bi_{W}}			
\newcommand{\BI}{\hat{\Bi}}			

\begin{document}

\title[A characterization of alternating link exteriors]{A characterization of alternating link exteriors in terms of cubed complexes}
\author{Shunsuke Sakai}
\address{Department of Mathematics\\
	Graduate School of Science\\
	Hiroshima University\\
	Higashi-Hiroshima, 739-8526, Japan}
\email{shunsuke463@gmail.com}

\begin{abstract}
	We give a characterization of alternating link exteriors in terms of cubed complexes.
	To this end, we introduce the concept of a ``signed BW cubed-complex'',
	and give a characterization for a signed BW cubed-complex to 
	have the underlying space which is homeomorphic to an alternating link exterior.
\end{abstract}

\maketitle

\section{Introduction}\label{Intro}
	Recently, Greene~\cite{17G} and Howie~\cite{17Ho}, independently, established
	intrinsic characterizations of alternating links in terms of a pair of spanning surfaces,
	answering an old question of R.\ H.\ Fox.
	These results can be regarded as characterizations of alternating link exteriors which have marked meridians
	(see \cite[Theorem 3.2]{17Ho}).

	The purpose of this paper is to give a characterization of alternating link exteriors
	from the viewpoint of cubed complexes.
	Our starting point is a cubical decomposition of alternating link exteriors,
	which is originally due to Aitchison, and is used by
	Agol~\cite{Agol}, Adams~\cite{Adams}, Thurston~\cite{DThurston}, Yokota~\cite{Yokota1, Yokota2}
	and Sakuma-Yokota~\cite{16SY}.
	Thus we call it the \textit{Aitchison complex}.
	The Aitchison complex for an alternating link is actually a mapping cylinder of the natural map
	from the boundary of the exterior of the alternating link onto the Dehn complex.
	For a detailed description and historical background, see \cite{16SY}.
	
	In this paper, we introduce the concepts of
	a \textit{signed BW squared-complex} (or an \textit{SBW squared-complex}, for short)
	and
	a \textit{signed BW cubed-complex} (or an \textit{SBW cubed-complex}, for short),
	and give a combinatorial description of the Dehn complex and the Aitchison complex
	as an SBW squared-complex and an SBW cubed-complex, respectively.
	The main theorem gives a necessary and sufficient condition
	for a given SBW cubed-complex to be isomorphic to the Aitchison complex of some alternating link exterior
	(Theorem~\ref{SCCA}).
	This implies a characterization of alternating link exteriors in terms of cubed complexes (Corollary~\ref{main-cor}).
	
	This paper is organized as follows.
	In Section~\ref{desc},
	we give an intuitive description of the Aitchison complex and the Dehn complex
	following~\cite{DThurston, Yokota1}.
	In Section~\ref{SC},
	we introduce the SBW squared-complex and the SBW cubed-complex,
	and describe the Dehn complex and the Aitchison complex
	in terms of the SBW squared-complex and the SBW cubed-complex, respectively.
	In Section~\ref{chara}, we prove the main theorem.

	The author would like to thank his supervisor, Makoto Sakuma, for valuable suggestions.
	He would also like to thank Naoki Sakata and Takuya Katayama for their support and encouragement.

\section{An intuitive description of the Aitchison complexes \\ and the Dehn complexes for alternating links} \label{desc}
	In this section,
	we give an intuitive description of the Aitchison complexes and the Dehn complexes
	following~\cite{DThurston, Yokota1}.
	For detailed description, see~\cite{16SY}.
	
	Let $\G \subset S^{2}$ be a connected alternating link diagram
	and $L\subset S^{3}$ the alternating link represented by $\G$.
	We pick two points $P_{+}$ and $P_{-}$ in the components of $S^{3}\setminus S^{2}$ one by one.
	These points are regarded to lie above and below $S^{2}$, respectively.
	Identify $S^{3}\setminus \{P_{+}, P_{-}\}$ with $S^{2}\times \R$, and assume the following.
	The diagram $\G$ is regarded as a $4$-valent graph in $S^{2}\times \{0\}$,
	$L\subset \G \times [-1, 1]\subset S^{2}\times [-1, 1]$,
	and $L$ intersects $S^{2}\times \{0\}$ transversely in $2n$ points,
	where $n$ is the crossing number of $\G$.
	
	For each vertex $x$ of $\G$,
	consider a square $s$ in $S^{2}=S^{2}\times \{0\}$
	which forms a relative regular neighborhood of $x$ in $(S^{2}, \G)$
	such that the four vertices of $s$ lie in the four germs of edges around $x$.
	Let $x^{+}$ and $x^{-}$ be the points of $L$ which lie above and below $x$, respectively.
	Consider the two pyramids, $\Delta^{\pm}$, in $S^{3}$
	which are obtained as the joins $x^{\pm}*s$.
	We may assume $\Delta^{\pm}\cap L=\{x^{\pm}\}$, $\Delta^{+}\cap \Delta^{-}=s$.
	Note that $\Delta^{+}\cup \Delta^{-}$ is an octahedron
	which contains the crossing arc of $L$ determined by $x$ (see Figure~\ref{crosses}(b)).
	Let $\{\Delta_{1}^{\pm}, \dots, \Delta_{n}^{\pm}\}$
	be the set of $2n$ pyramids in $S^{3}$ located around the vertices of $\G$.
	
	Pick an edge $e$ of the graph $\G$,
	and let $x_{1}$ and $x_{2}$ be the vertices of $\G$ joined by $e$,
	such that the arc $\tilde{e}=L\cap (e\times [-1, 1])$ in $L$
	joins $x_{1}^{+}$ and $x_{2}^{-}$ (see Figure~\ref{crosses}(a)).
	Let $a_{i}$ be the vertex of $s_{i}$ contained in $e$ ($i = 1,2$).
	Let $R$ be one of the two regions of $\G$ in $S^{2}$ whose boundary contains the edge $e$,
	and let $b_{i}$ be the vertex of $s_{i}$ 
	such that the edge $a_{i}b_{i}$ of $s_{i}$ is contained in $R$.
	\begin{figure}
		\begin{listliketab}
		\begin{longtable}[c]{rcp{0.6\textwidth}}
		(a) & \raisebox{-\height}{\includegraphics[width=180pt]{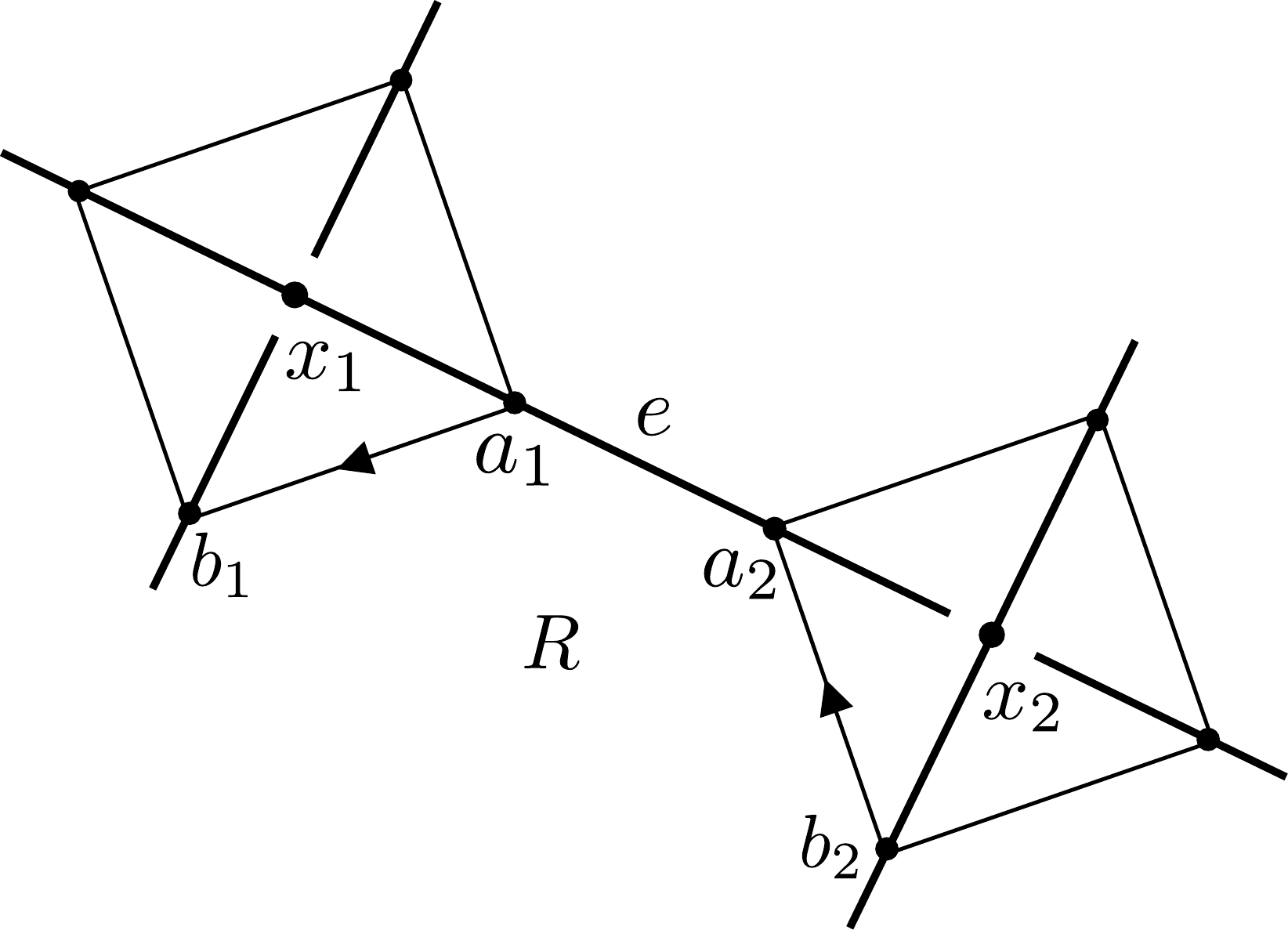}}\\

		(b) & \raisebox{-\height}{\includegraphics[width=230pt]{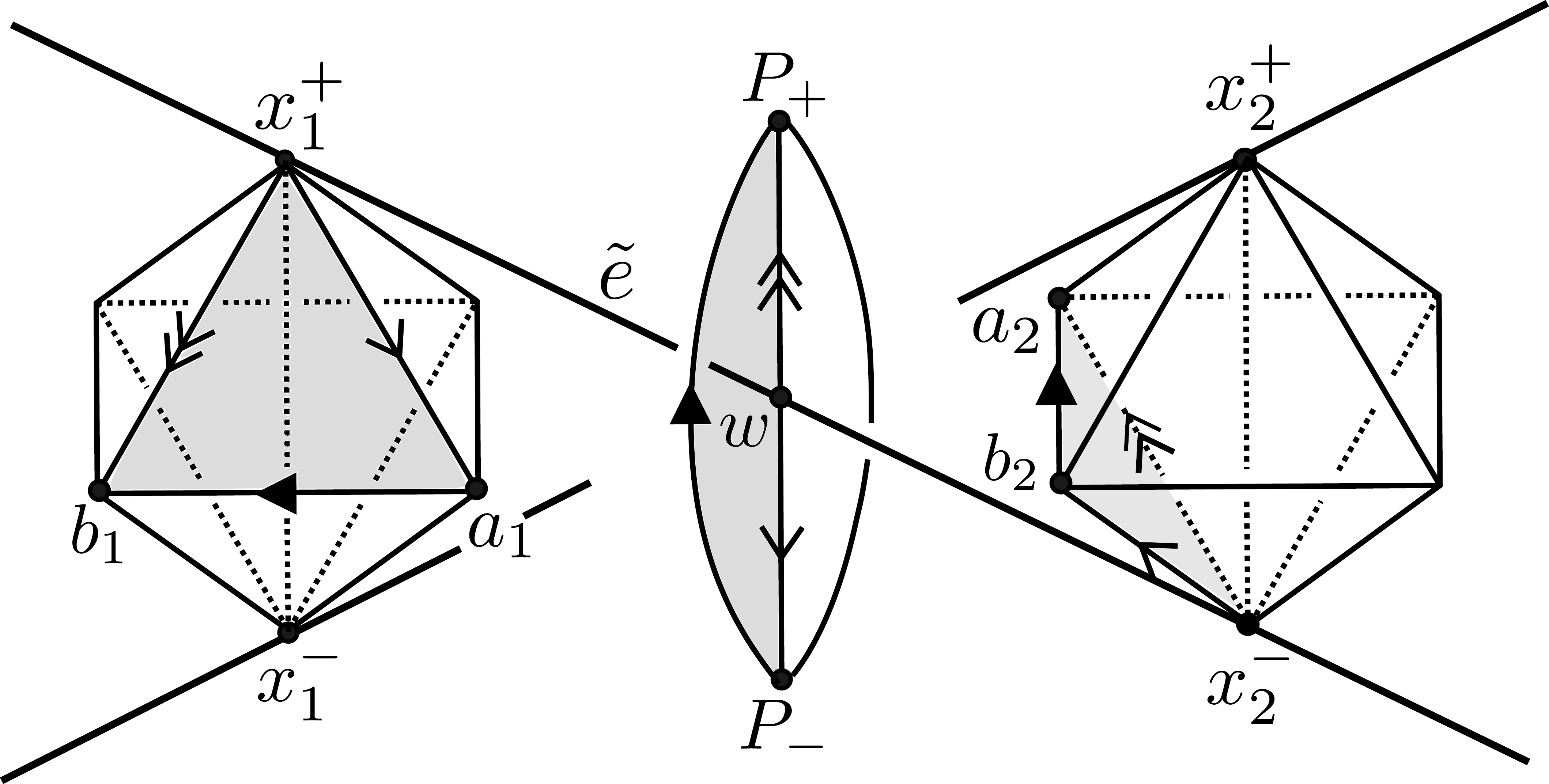}}
		\end{longtable}
		\end{listliketab}
	\caption{}
	\label{crosses}
	\end{figure}
	Let $w = \tilde{e} \cap S^{2}$ be the ``middle point'' of $\tilde{e}$,
	and consider the vertical line segments $wP_{+}$ and $wP_{-}$.
	Then we have the following relative isotopies in $(S^{3}, L)$ (see Figure \ref{crosses}(b)).
	\begin{enumerate}
		\item The edges $x_{1}^{+}a_{1}$ and $x_{1}^{+}b_{1}$ of the pyramid $\Delta_{1}^{+}$
			are isotopic to the vertical line segments $wP_{-}$ and $wP_{+}$, respectively.
		\item The edges $x_{2}^{-}a_{2}$ and $x_{2}^{-}b_{2}$ of the pyramid $\Delta_{2}^{-}$
			are isotopic to the vertical line segments $wP_{+}$ and $wP_{-}$, respectively.
		\item The edge $a_{1}b_{1}$ of $s_{1}$ is isotopic to an almost vertical line segment
			which starts $P_{-}$, passes through a point in the interior of $R$ and reaches $P_{+}$.
		\item The edge $b_{2} a_{2}$ of $s_{2}$ is isotopic to an almost vertical line segment
			which starts $P_{-}$, passes through a point in the interior of $R$ and reaches $P_{+}$.
	\end{enumerate}
	These isotopies determine an isotopy (and so a homeomorphism)
	from the face $x_{1}^{+}a_{1}b_{1}$ of $\Delta_{1}^{+}$
	onto the face $x_{2}^{-}b_{2}a_{2}$ of $\Delta_{2}^{-}$.
	In this way, we obtain a paring of faces of
	the octahedra $\{\Delta_{i}^{+}\cup \Delta_{i}^{-}\}_{i}$
	and a homeomorphism between the faces in each pair.
	Let $O_{i}^{\pm}$ be the cubes obtained from the pyramids $\Delta_{i}^{\pm}$
	by chopping off a small regular neighborhoods of $x_{i}^{\pm}$.
	Then the above pairing and homeomorphisms determine
	a gluing information for the cubes $\{O_{i}^{\pm}\}_{i}$.
	Let $\A$ be the resulting cubed complex and $\D$ the subcomplex of $\A$ obtained
	by gluing the squares $\{\Delta_{i}^{+}\cap \Delta_{i}^{-}\}_{i}$.
	Then we have the following.
	\begin{prop}
		For a connected alternating diagram $\G$,
		$\A$ gives a cubical decomposition of the exterior $E(L)$ of the link $L$ represented by $\G$.
		Moreover, there is a deformation retraction of $\A$ onto $\D$,
		and so, $\D$ is a spine of $E(L)$.
	\end{prop}
	In fact, $\D$ is isotopic to the Dehn complex of the diagram $\G$.
	(For the definition of the Dehn complex, see~\cite{BH, 13BH, Wise}.)
	We call $\A$ the \textit{Aitchison complex} of $\G$.

\section{Signed BW complexes}\label{SC}
	In this section,
	we introduce the concept of a \textit{signed BW squared-complex}
	and that of a \textit{signed BW cubed-complex},
	and then describe the Dehn complexes and the Aitchison complexes for alternating links by using these concepts.

	By a \textit{signed BW square} (or an \textit{SBW-square}, for short),
	we mean the square $s:=[0,1]^{2}$ with the following information:
	\begin{enumerate}
		\item The vertices $(0, 0)$ and $(1, 1)$ are endowed with the sign $-$,
			and the vertices $(0, 1)$ and $(1, 0)$ are endowed with the sign $+$.
		\item The horizontal edges $I\times \{0\}$ and $I\times \{1\}$ are endowed with the color $B$ (Black),
			and the vertical edges $\{0\} \times I$ and $\{1\} \times I$ are endowed with the color $W$ (White).
	\end{enumerate}
	For an SBW-square $s$, we assume that each edge of $s$ is oriented
	so that the initial point and the terminal point have the sign $-$ and $+$, respectively.
	Now consider a set $S=\{s_{1}, \dots, s_{n}\}$ of $n$ copies of the SBW-square,
	and let $V_{+}(S)$ and $V_{-}(S)$, respectively,
	be the sets of positive vertices and negative vertices of the SBW-squares in $S$.
	
	For each bijection $\bi \colon V_{+}(S)\to V_{-}(S)$,
	we construct a squared complex (i.e.\ two-dimensional cubed complex), $\SSCS$, as follows.
	Let $E_{B}(S)$ and $E_{W}(S)$, respectively,
	be the set of the black edges and the white edges of the SBW-squares in $S$.
	Then $\bi$ induces a bijection $\Bib \colon E_{B}(S)\to E_{B}(S)$ as follows.
	For a black edge $e \in E_{B}(S)$,
	let $v$ be the positive vertex which forms the terminal point of $e$.
	Then $\Bib(e)$ is defined to be the unique black edge whose initial vertex is $\bi(v)$ (see Figure~\ref{scsquare2}).
	\begin{figure}
		\includegraphics[width=200pt]{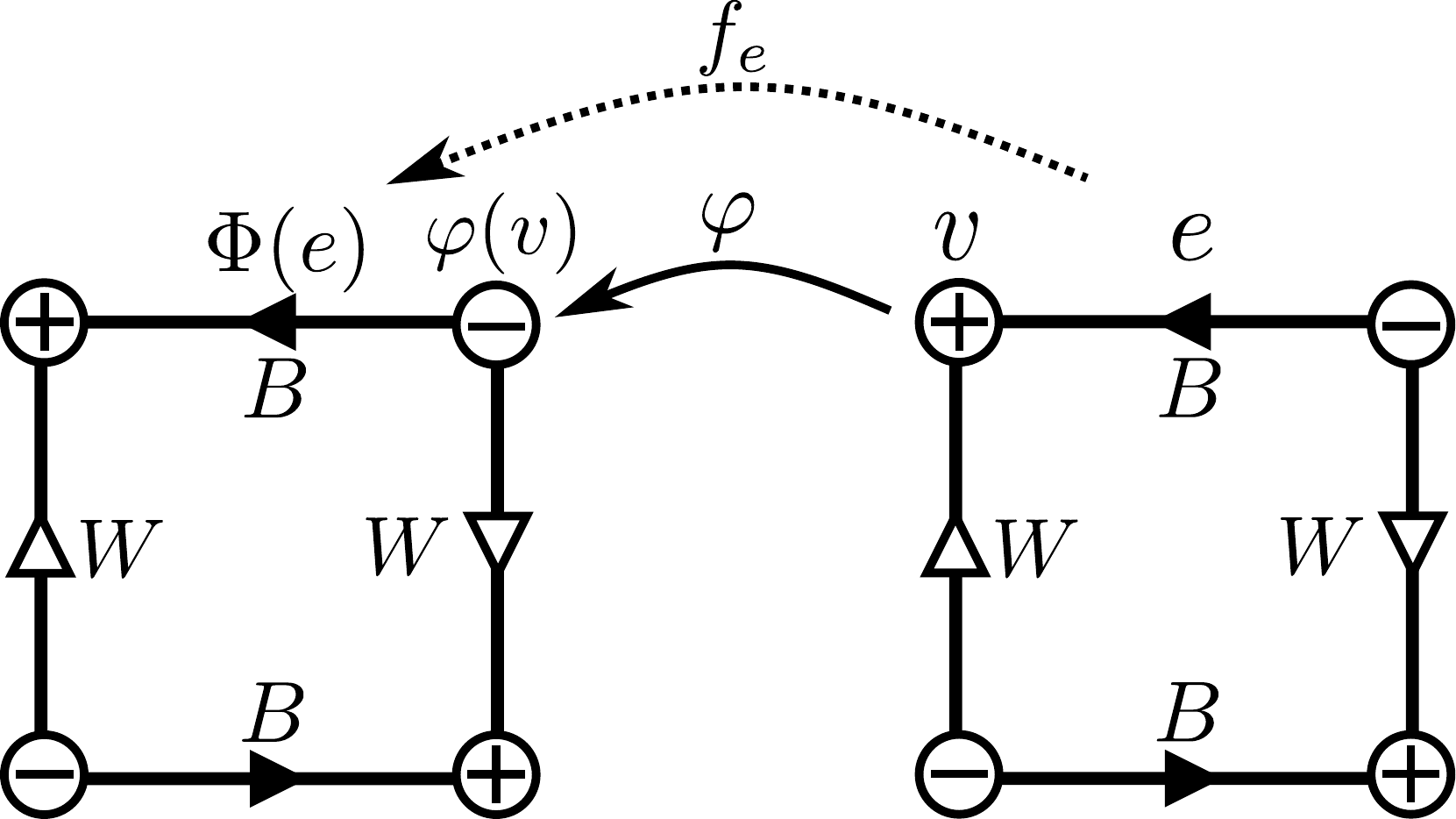}
		\caption{}
		\label{scsquare2.pdf}
	\end{figure}
	Similarly, $\bi$ induces a bijection $\Biw \colon E_{W}(S)\to E_{W}(S)$,
	such that $\Biw(e)$, for $e\in E_{W}(S)$, is the unique white edge whose initial vertex
	is the image of the terminal vertex of $e$ by $\bi$.
	Thus we obtain a bijection $\Bi:=\Bib \sqcup \Biw$ from $E(S):=E_{B}(S)\sqcup E_{W}(S)$ to itself.
	For each $e \in E(S)$, let $f_{e}\colon e\to \Bi(e)$ be the unique orientation-preserving linear homeomorphism.
	We regard the family $\{f_{e}\colon e\to \Bi(e)\}_{e\in E(S)}$ as a gluing information
	for the SBW-squares $S=\{s_{1}, \dots, s_{n}\}$,
	and denote the resulting squared complex by $\SSCS$.
	We call it the \textit{signed BW squared-complex} (or the \textit{SBW squared-complex}, for short)
	determined by the bijection $\bi \colon V_{+}(S)\to V_{-}(S)$.

	\begin{rem}\upshape
		A signed BW squared-complex is a special case of a VH-complex introduced by Wise \cite{Wise},
		which is defined to be a squared complex
		whose edges are partitioned into two classes V (vertical) and H (horizontal).
		Motivated by black/white checkerboard surfaces, 
		we use W and B, instead of V and H.
	\end{rem}
	
	For the SBW squared-complex $\SSCS$,
	we define the associated SBW cubed-complex, $\SSCC$, as follows.
	For the set $S=\{s_{1}, \dots, s_{n}\}$ of the SBW-squares,
	consider the set of the ``upper SBW-cubes'' $\{s_{i}\times [0,1]\}_{i=1}^{n}$
	and the ``lower SBW-cubes'' $\{s_{i}\times [-1,0]\}_{i=1}^{n}$.
	Consider also the set of ``upper side-faces" $F_{+}:=\{e\times [0,1]\}_{e\in E(S)}$
	and the set of ``lower side-faces" $F_{-}:=\{e\times [-1,0]\}_{e\in E(S)}$.
	Then the bijection $\Bi \colon E(S)\to E(S)$ induces the bijection
	$\BI \colon F_{-}\to F_{+}$ defined by $\BI(e\times [-1, 0])=\Bi(e)\times [0,1]$.
	Moreover, the linear homeomorphism $f_{e}\colon e\to \Bi(e)$ induces
	the linear homeomorphism $\hat{f}_{e}\colon e\times [-1,0] \to \Bi(e)\times [0,1]$
	defined by $\hat{f}_{e}(x,t)=(f_{e}(x),-t)$.
	By gluing the side-faces of the cubes $\{e\times [-1,0]\cup e\times [0,1]\}_{e\in E(S)}$
	by the family of homeomorphisms
	$\{\hat{f}_{e}\colon e\times [-1,0]\to \Bi(e)\times [0,1]\}_{e\in E(S)}$,
	we obtain a three-dimensional cubed complex.
	We denote it by $\SSCC$,
	and call it the \textit{signed BW cubed-complex} (or the \textit{SBW cubed-complex}, for short)
	determined by $\bi$.
	It should be noted that $\SSCS$ is a subcomplex of $\SSCC$,
	and there is a natural deformation retraction of $\SSCC$ onto $\SSCS$.

	For an alternating link $L \subset S^{3}$ represented by a connected alternating diagram $\G \subset S^{2}$,
	the Dehn complex $\D$ and the Aitchison complex $\A$ are identified with
	the SBW squared-complex $\SSCS$ and the SBW cubed-complex $\SSCC$, respectively,
	where $S$ and $\bi$ are defined as follows.
	Consider the checkerboard coloring of $(S^{2}, \G)$
	such that the associated black surface for $L$ has a positive half-twist at each crossing (see Figure~\ref{fig8}(a)).

	\begin{figure}
		\centering
		\begin{subfigure}{0.3\columnwidth}
			\centering
			\includegraphics[width=\columnwidth]{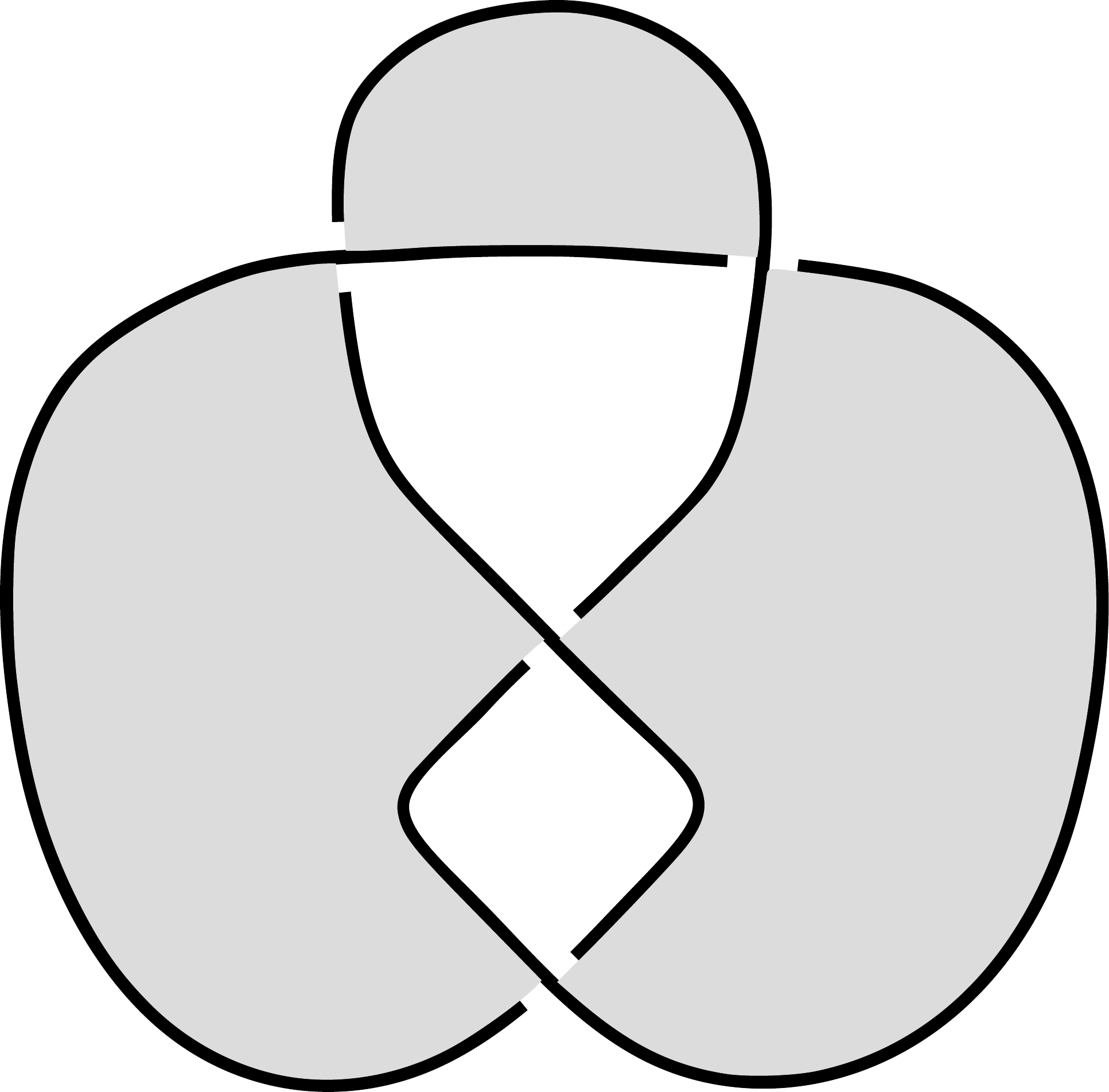}
			\caption*{(a)}
		\end{subfigure}
		\hspace{15mm}
		\begin{subfigure}{0.3\columnwidth}
			\centering
			\includegraphics[width=\columnwidth]{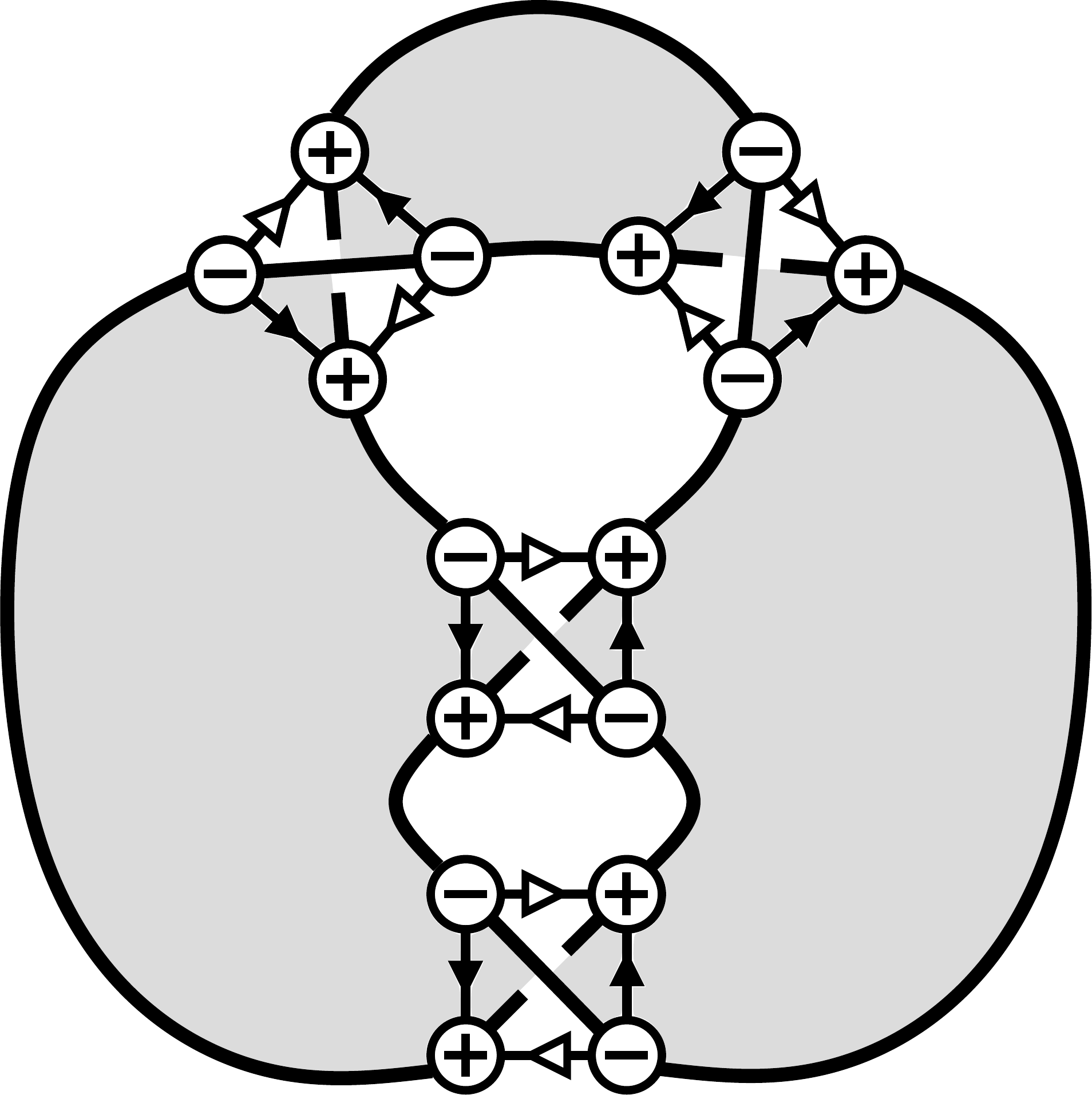}
			\caption*{(b)}
		\end{subfigure}
		\caption{}
		\label{fig8}
	\end{figure}
	Let $\{x_{1}, \dots, x_{n}\}$ be the vertex set of $\G$
	and let $S =\{s_{1}, \dots, s_{n}\}$ be the set of SBW-squares as illustrated in Figure~\ref{fig8}(b).
	To be precise,
	\begin{enumerate}
		\item $s_{i}$ is a square in $S^{2}$ which forms a relative regular neighborhood of $x_{i}$
			in $(S^{2}, \G)$, and the vertices of $s_{i}$ are contained in $\G$.
		\item Each vertex of $s_{i}$ has the sign $+$ or $-$
			according to whether it lies in an underpass or an overpass.
		\item Each edge of $s_{i}$ is colored $B$ or $W$
			according to whether it lies in a black region or a white region.
	\end{enumerate}
	Observe that $\G \setminus \bigcup_{i=1}^{n}\inter (s_{i})$ is a disjoint union of arcs
	and that the boundary of each of which consists of a vertex in $V_{+}(S)$ and a vertex in $V_{-}(S)$.
	This determines a bijection $\bi \colon V_{+}(S)\to V_{-}(S)$.
	Then the following proposition is obvious from the construction of the Aitchison complex.
	\begin{prop}\label{SCSD}
		Under the above setting,
		the SBW squared-complex $\SSCS$ and the SBW cubed-complex $\SSCC$ are isomorphic to
		the Dehn complex $\D$ and the Aitchison complex $\A$, respectively.
	\end{prop}

\section{Main results}\label{chara}
	Proposition~\ref{SCSD} shows that the Aitchison complex $\A$ of a connected alternating diagram $\G$
	can be described as the SBW cubed-complex $\SSCC$.
	In this section, we prove Theorem~\ref{SCCA},
	which gives a characterization of the Aitchison complexes of connected alternating diagrams
	among the SBW cubed-complexes.
	
	Let $S=\{s_{1}, \dots, s_{n}\}$ be a set of SBW-squares,
	and let $\bi \colon V_{+}(S)\to V_{-}(S)$ be a bijection,
	where $V_{\pm}(S)$ are the sets of positive and negative vertices of the SBW-squares in $S$.
	Let $\Bi =\Bib \sqcup \Biw$ be the bijection from $E(S)=E_{B}(S)\sqcup E_{W}(S)$ to itself
	determined by $\bi$.
		
	\begin{thm}\label{SCCA}
		Under the above setting,
		the SBW cubed-complex $\SSCC$ is isomorphic to the Aitchison complex $\A$
		of a connected alternating diagram $\G$,
		if and only if the bijection $\Bi$ satisfies
		\begin{equation*}
			|E(S)/\langle \Bi \rangle |=|S|+2,
		\end{equation*}
		where $E(S)/\langle \Bi \rangle$ denotes the quotient space of the cyclic group action
		on $E(S)$ induced by $\Bi$,
		and $|\cdot|$ denotes the cardinality of a set.
	\end{thm}

	\begin{proof}
		We first prove the only if part.
		Suppose that an SBW cubed-complex $\SSCC$ is isomorphic to
		the Aitchison complex $\A$ of a connected alternating diagram $\G \subset S^{2}$.
		Then we may assume $S$ and $\bi$ are constructed from $\G$ as in Section~\ref{SC}.
		Observe that there is a one-to-one correspondence between
		$E_{B}(S)/\langle \Bi \rangle$ (resp.\ $E_{W}(S)/\langle \Bi \rangle$)
		and the set of the black (resp.\ white) regions of $\G$.
		Consider the cell decomposition of the projection plane $S^{2}$ obtained from $\G$.
		Then the above observation implies that
		$|E(S)/\langle \Bi \rangle |$ is equal to the number of 2-cells of the cell decomposition.
		Since the cell decomposition has $n$ vertices and each vertex has degree four,
		the number of $1$-cells is equal to $2n$ when $n = |S|$.
		Hence, we have
		\[2=\chi(S^{2})=n-2n+|E(S)/\langle \Bi \rangle |.\]
		This implies $|E(S)/\langle \Bi \rangle |=|S|+2$,
		completing the proof of the only if part.
		
		Next, we prove the if part.
		Suppose $|E(S)/\langle \Bi \rangle |=|S|+2$.
		By using this condition, we construct a connected alternating diagram $\G$ such that $\SSCC \cong \A$.
		Consider the two-dimensional complex, $X$,
		obtained from the set $S=\{s_{1}, \dots, s_{n}\}$ of SBW-squares
		by attaching a 1-cell $\gamma =\langle v, \bi(v) \rangle$ for each $v\in V_{+}(S)$,
		and we now attach black/white 2-cells to $X$, as follows.
		
		Consider the action of the cyclic group $\langle \Bib \rangle$ on $E_{B}(S)$,
		and pick its orbit $\{e, \Bib(e), \dots, \Bib^{k-1}(e)\}$ with $\Bib^{k}(e)=e$.
		Then for each $i\in \{0, \dots, k-1\}$,
		the terminal vertex of $\Bib^{i}(e)$ is mapped by $\bi$ to the initial vertex of $\Bib^{i+1}(e)$,
		and so there is an edge, $\gamma_{i}$, of $X$ joining these two points.
		Then,
		\[e+\gamma_{0}+\Bib(e)+\gamma_{1}+\dots +\Bib^{k-1}(e)+\gamma_{k-1}\]
		determines a simple 1-cycle, where $\gamma_{i}$ is given a natural orientation.
		We attach a black 2-cell to $X$ along the 1-cycle.
		Similarly, each orbit of the action of $\langle \Biw \rangle$ on $E_{W}(S)$ determines a simple 1-cycle,
		and we attach a white 2-cell to $X$ along the 1-cycle.

		Let $M$ be the two-dimensional cell complex obtained from $X$ by attaching black/white 2-cells in this way.
		We can easily observe that $M$ is an orientable $2$-manifold.
		To compute the Euler characteristic $\chi(M)$, observe the following.
		\begin{enumerate}
			\item The number of vertices of $M$ is $4n$ with $n=|S|$,
				since each vertex is contained in an SBW-square.
			\item The edge set of $M$ consists of $4n$ edges of the SBW-squares
				and $2n$ ``connecting'' edges.
				So, the number of edges of $M$ is equal to $6n$.
			\item The face set of $M$ consists of $n$ squares
				and $|E_{B}(S)/\langle \Bib \rangle |$ black 2-cells
				and $|E_{W}(S)/\langle \Biw \rangle |$ white 2-cells.
				So, the number of faces of $M$ is $n + | E(S)/\langle \Bi \rangle |$.
		\end{enumerate}
		Hence, \[\chi(M)=4n-6n+(n+|E(S)/\langle \Bi \rangle |)=-n+|E(S)/\langle \Bi \rangle |, \]
		and so, by the assumption, it is equal to $2$.
		Therefore, $M$ is homeomorphic to $S^{2}$.
		
		Add to an overpass connecting two negative vertices and an underpass connecting two positive vertices
		in each SBW-square.
		The union of connecting edges and overpasses and underpasses
		gives a connected link diagram $\G$, which is clearly alternating.
		Moreover, it is obvious from the construction that $\SSCC$ is isomorphic to $\A$.
	\end{proof}

	\begin{cor}\label{main-cor}
		A compact $3$-manifold $M$ is homeomorphic to the exterior of an alternating link $L$
		represented by a connected alternating diagram,
		if and only if $M$ is homeomorphic to the underlying space of an SBW cubed-complex $\SSCC$
		such that $|E(S)/\langle \Bi \rangle |=|S|+2$.
	\end{cor}

	\begin{rem}\upshape
		If the identity in Theorem~\ref{SCCA} is not satisfied,
		then the surface $M$ in the proof is a closed orientable surface of genus $\ge 1$,
		and $\G$ is an alternating diagram in the surface $M$.
		In this case the underlying space of $\A$ is homomorphic to
		the \textit{Dehn space} of the link $L$ in $M\times [-1, 1]$
		represented by the diagram $\G$,
		namely the space obtained from the exterior of $L$
		by coning off $M \times \{\pm 1\}$ (see~\cite{13BH}).
		We note that the ``Dehn complexes'' of these spaces and related spaces are studied extensively
		in the works \cite{03HR, 12Ha, 13BH}
		by Harlander, Rosebrock, and Byrd.
	\end{rem}

\end{document}